\documentclass[11pt]{amsart}
\usepackage[centertags]{amsmath}
\usepackage{amsfonts}
\usepackage{amssymb}
\usepackage{amsthm}
 \newtheorem{thm}{Theorem}[section]

 \newtheorem{prop}[thm]{Proposition}
 \theoremstyle{definition}
 
 \theoremstyle{remark}
 \newtheorem{rem}[thm]{Remark}
 \numberwithin{equation}{section}
\begin{document}
\title[Classification of 
Lie Bialgebra Structures over Polynomials]{On the General Classification of 
Lie Bialgebra Structures over Polynomials}

\author{IULIA POP}
\author{JULIA YERMOLOVA--MAGNUSSON}
\address{Department of Mathematical Sciences, University of Gothenburg, 
Sweden. Email: iulia@chalmers.se; md1jm@chalmers.se}

\begin{abstract}
  The present paper is a continuation of \cite{PM}, where Lie bialgebra structures on $\mathfrak{g}[u]$ were studied. These structures fall into different classes labelled by the vertices of the extended Dynkin diagram of $\mathfrak{g}$. In \cite{PM} the Lie bialgebras corresponding to $-\alpha_{\rm{max}}$ were classified. In the present article, we investigate the Lie bialgebras corresponding to an arbitrary simple root $\alpha$.
 \end{abstract}
\maketitle

\section{Introduction}
Let $\mathfrak{g}$ denote a finite dimensional simple complex Lie algebra. 
The aim of the present article is to further investigate Lie bialgebra structures on $\mathfrak{g}[u]$ in the framework provided in \cite{SZ,PM}. 
As it was shown in \cite{SZ}, any Lie bialgebra structure $\delta$ on $\mathfrak{g}[u]$ can be extended to a Lie bialgebra structure $\bar{\delta}$ on $\mathfrak{g}[[u]]$. The classical double associated to any such structure is isomorphic to one of the following Lie algebras:
 
\textbf{Case I.} $\mathfrak{g}((u))$ together with the following nondegenerate bilinear form
\[Q_{a(u)}(f_1(u),f_2(u))=\mathrm{Res}_{u=0}(a(u)K(f_1(u),f_2(u))),\]
where $K$ is the Killing form of the Lie algebra $\mathfrak{g}((u))$ over 
$\mathbb{C}((u))$. Here $a(u)=1+\sum_{k=0}^{\infty}a_ku^k$ is a Taylor series 
which satisfies certain properties which will be presented later.

\textbf{Case II.} $\mathfrak{g}((u))\oplus \mathfrak{g}$, endowed with the following nondegenerate bilinear form:
\[ Q_{a(u)}(f_1(u)+x_1,f_2(u)+x_2)=\mathrm{Res}_{u=0}(u^{-1}a(u)K(f_1(u),f_2(u)))-K(x_1,x_2),\]
for all $f_1(u), f_2(u)\in\mathfrak{g}((u))$ and $x_1,x_2\in\mathfrak{g}$. Again $a(u)=1+\sum_{k=0}^{\infty}a_ku^k$ is a certain Taylor series.

\textbf{Case III.} $\mathfrak{g}((u))\oplus (\mathfrak{g}+\varepsilon\mathfrak{g})$, where $\varepsilon^2=0$. Here we consider \[Q_{a(u)}(f_1(u)+x_2+\varepsilon x_3,f_2(u)+y_2+\varepsilon y_3)=\mathrm{Res}_{u=0}(u^{-2}a(u)K(f_1(u),f_2(u))-\]
\[-K(x_3,y_2)-K(x_2,y_3),\]
for any $f_1(u),f_2(u)\in\mathfrak{g}((u))$ and 
$x_2,x_3,y_2,y_3\in\mathfrak{g}$, and $a(u)=1+\sum_{k=0}^{\infty}a_ku^k$ is a certain Taylor series.

The Lie bialgebra structure $\bar{\delta}$ (and implicitly $\delta$) is uniquely defined by a bounded Lagrangian subalgebra $W$ of $\mathfrak{g}((u))$, $\mathfrak{g}((u))\oplus \mathfrak{g}$ or $\mathfrak{g}((u))\oplus (\mathfrak{g}+\varepsilon\mathfrak{g})$. Moreover, $W$ should be transversal to $\mathfrak{g}[[u]]$.

The classification of bialgebra structures is reduced in this way to the classification of bounded Lagrangian subalgebras $W$ with the above property. Moreover, any such $W$ can be embedded into a special algebra, a so-called maximal order, indexed by vertices of the extended Dynkin diagram of $\mathfrak{g}$. We recall its construction, according to \cite{S}. Let $\mathfrak{h}$ be a Cartan subalgebra of $\mathfrak{g}$ with
 the corresponding set of roots $R$ and a choice of simple
 roots $\Gamma$. Denote by $\mathfrak{g}_{\alpha}$ the root space
 corresponding to a root $\alpha$. Let $\mathfrak{h}(\mathbb{R})=\{h\in\mathfrak{h}: \alpha(h)\in\mathbb{R},\alpha\in R\}$.
 Consider the valuation on $\mathbb{C}((u^{-1}))$ defined by
 $v(\sum_{k\geq n}a_{k}u^{-k})=n$. For any root $\alpha$ and any
 $h\in\mathfrak{h}(\mathbb{R})$, set
 $M_{\alpha}(h)$:=$\{f\in\mathbb{C}((u^{-1})):v(f)\geq \alpha(h)\}$.
 Consider
 \[\mathbb{O}_{h}:=\mathfrak{h}[[u^{-1}]]\oplus(\oplus_{\alpha\in R}M_{\alpha}(h)\otimes\mathfrak{g}_{\alpha}).\]
Let $\{h\in\mathfrak{h}(\mathbb{R}): \alpha(h)\geq 0, \alpha\in \Gamma, \alpha_{\max}(h)\leq 1\}$ be the standard simplex. Vertices of the above simplex correspond to vertices of the extended Dynkin diagram of
$\mathfrak{g}$, the correspondence being given by the following rule:
$0\leftrightarrow\ -\alpha_{\max}, h_{i}\leftrightarrow\alpha_{i}$, 
where $\alpha_{i}(h_{j})=\delta_{ij}/k_{j}$, and $k_{j}$ are given by the relation $\sum
k_{j}\alpha_{j}=\alpha_{\max}$. One writes $\mathbb{O}_{\alpha}$ instead of
$\mathbb{O}_{h}$ if $\alpha$ is the root which corresponds to the vertex $h$, 
and $\mathbb{O}_{-\alpha_{\max}}$ instead of $\mathbb{O}_0$.

According to \cite{SZ}, in case I, $1/a(u)$ is a polynomial of degree at most 2, in case II, is at most 1 and in case III, its degree is 0. Moreover, by means of a change of variable in $\mathbb{C}[u]$ and rescaling the nondegenerate bilinear form $Q_{a(u)}$, one may assume that $a(u)$ has one of the following forms: 
\begin{enumerate}
\item
$a(u)=1/(1-c_1u)(1-c_2u)$, for non-zero constants $c_1\neq c_2$ 

\item $a(u)=1/(1-u)^2$ 

\item $a(u)=1/1-u$ 

\item $a(u)=1$. 

\end{enumerate}

In what follows we will treat separately case I (subcases 1--4), case II (subcases 3--4 are the only possible), case III (subcase 4 is the only possible).

\section{Lie bialgebra structures on $\mathfrak{g}[u]$ in case I}

Here we consider $\mathfrak{g}((u))$ together with the form 
\[Q_{a(u)}(f_1(u),f_2(u))=\mathrm{Res}_{u=0}(a(u)K(f_1(u),f_2(u))),\]
for any $f_1(u), f_2(u)\in\mathfrak{g}((u))$. 
In \cite{SZ} the following result was proved: 
\begin{prop}
\cite{SZ} There exists a one-to-one correspondence between 
Lie bialgebra structures 
$\delta$ on $\mathfrak{g}[u]$ satisfying  
$D_{\bar{\delta}}(\mathfrak{g}[[u]])=\mathfrak{g}((u))$ and 
bounded Lagrangian subalgebras $W$ of $\mathfrak{g}((u))$, with respect to 
the nondegenerate bilinear form $Q_{a(u)}$, 
and transversal to $\mathfrak{g}[[u]]$.
\end{prop}

In \cite{PM}, we treated this infinite-dimensional problem by first showing that any such $W$ can be embedded into a special algebra, a so-called maximal order, indexed by vertices of the extended Dynkin diagram of $\mathfrak{g}$. This embedding allows us to replace our infinite-dimensional problem with a finite-dimensional one.  
\begin{prop}\label{max_ord}\cite{PM}
Suppose that $W$ is a bounded Lagrangian subalgebra of $\mathfrak{g}((u))$, with respect to $Q_{a(u)}$ and transversal to $\mathfrak{g}[[u]]$. Then there exists $\sigma\in\mathrm{Aut}_{\mathbb{C}[u]}(\mathfrak{g}[u])$ such that $\sigma(u)(W)\subseteq \mathbb{O}_{\alpha}\cap \mathfrak{g}[u,u^{-1}]$, where $\alpha$ is either a simple root or $-\alpha_{\max}$. 
\end{prop}

In \cite{PM} the Lie bialgebras corresponding to $-\alpha_{\rm{max}}$ were classified. In the present article, we investigate the Lie bialgebras corresponding to an arbitrary simple root $\alpha$.

Consider the Lie algebra $\mathfrak{g}\oplus\mathfrak{g}$, together with
the nondegenerate bilinear form 
\[\bar{Q}((x_1,y_1),(x_2,y_2))=K(x_1,x_2)-K(y_1,y_2),\] 
for any elements $x_1$, $y_1$, $x_2$, $y_2\in \mathfrak{g}$. Let us fix an arbitrary simple root $\alpha$. Denote by $P_{\alpha}^{-}$ the standard parabolic subalgebra of $\mathfrak{g}$ spanned by the root spaces of all negative roots and all positive roots which do not contain $\alpha$. Let us denote by  $\Delta_{\alpha}$ the set of pairs $(x,y)\in P_{\alpha}^{-}\times P_{\alpha}^{-}$ which have the same reductive parts. 
Then the following result holds: 

\begin{thm}\label{A1}
Let $\alpha$ be a simple root and $k$ its coefficient in the decomposition of 
$\alpha_{\rm{max}}$. Let $a(u)=1/(1-c_1u)(1-c_2u)$, for non-zero constants $c_1\neq c_2$. 

(i) If $k=1$, there exists a one-to-one correspondence between Lagrangian subalgebras $W$ of $\mathfrak{g}((u))$, with respect to $Q_{a(u)}$, which are transversal to $\mathfrak{g}[[u]]$ and satisfy 
$W \subseteq \mathbb{O}_{\alpha}\cap\mathfrak{g}[u,u^{-1}]$, and Lagrangian subalgebras in $\mathfrak{g}\oplus\mathfrak{g}$, with respect to $\bar{Q}$, transversal to $\Delta_{\alpha}$. 

(ii) If $k>1$ then there are no Lagrangian subalgebras $W$ of $\mathfrak{g}((u))$, with respect to $Q_{a(u)}$, which are transversal to $\mathfrak{g}[[u]]$ and satisfy $W \subseteq \mathbb{O}_{\alpha}\cap\mathfrak{g}[u,u^{-1}]$. 
\end{thm}

\begin{proof}
Part (ii) was proved in \cite{SZ}. We will prove (i). Assume that $\alpha$ has coefficient 1 in $\alpha_{\rm{max}}$. Then the corresponding maximal order $\mathbb{O}_{\alpha}$ is  given by the formula
\[\mathbb{O}_{\alpha}=u^{-1}\mathfrak{g}_1[[u^{-1}]]+\mathfrak{g}_0[[u^{-1}]]+
u\mathfrak{g}_{-1}[[u^{-1}]],\]
where $\mathfrak{g}_1$ is the sum of the root spaces of positive roots which contain $\alpha$ (with coefficient 1),  $\mathfrak{g}_{-1}$ is the sum of the root spaces of negative roots which contain $\alpha$ (with coefficient 1), and $\mathfrak{g}_0$ consists of $\mathfrak{h}$ and the root spaces of all roots which do not contain  $\alpha$. 

Then  $\mathbb{O}_{\alpha}\cap\mathfrak{g}[u,u^{-1}]=u^{-1}\mathfrak{g}_1[u^{-1}]+\mathfrak{g}_0[u^{-1}]+
u\mathfrak{g}_{-1}[u^{-1}]$ and 
$(\mathbb{O}_{\alpha}\cap\mathfrak{g}[u,u^{-1}])^{\perp}=(u^{-1}-c_1)(u^{-1}-c_2)(u^{-1}\mathfrak{g}_1[u^{-1}]+\mathfrak{g}_0[u^{-1}]+
u\mathfrak{g}_{-1}[u^{-1}])$. 

Let us construct an isomorphism $\bar{\phi}$ between 
$\frac{\mathbb{O}_{\alpha}\cap\mathfrak{g}[u,u^{-1}]}{(\mathbb{O}_{\alpha}\cap\mathfrak{g}[u,u^{-1}])^{\perp}}$ and $\mathfrak{g}\oplus\mathfrak{g}$. We consider $\phi: \mathbb{O}_{\alpha}\cap\mathfrak{g}[u,u^{-1}]\longrightarrow \mathfrak{g}\oplus\mathfrak{g}$ given by 
$\phi(u^{-1}p_1(u^{-1})+p_0(u^{-1})+up_{-1}(u^{-1}))=(c_1p_1(c_1)+p_0(c_1)+c_1^{-1}p_{-1}(c_1), c_2p_1(c_2)+p_0(c_2)+c_2^{-1}p_{-1}(c_2))$, 
where $p_1\in \mathfrak{g}_1[u^{-1}]$, $p_0\in\mathfrak{g}_0[u^{-1}]$ and 
$p_{-1}\in\mathfrak{g}_{-1}[u^{-1}]$. One can easily check that the kernel of this map is exactly $(\mathbb{O}_{\alpha}\cap\mathfrak{g}[u,u^{-1}])^{\perp}$. Moreover, $\phi$ is surjective. Indeed, for any elements $a$, $b$ of 
$\mathfrak{g}$, let us uniquely decompose $a=a_1+a_0+a_{-1}$, $b=b_1+b_0+b_{-1}$with $a_1,b_1\in\mathfrak{g}_1$, $a_0,b_0\in\mathfrak{g}_0$, $a_{-1},b_{-1}\in\mathfrak{g}_{-1}$. Then one can find first degree polynomials $p_1\in \mathfrak{g}_1[u^{-1}]$, $p_0\in\mathfrak{g}_0[u^{-1}]$ and 
$p_{-1}\in\mathfrak{g}_{-1}[u^{-1}]$ such that $c_1p_1(c_1)=a_1$, 
$c_2p_1(c_2)=b_1$; $p_0(c_1)=a_0$, $p_0(c_2)=b_0$; 
$c_1^{-1}p_{-1}(c_1)=a_{-1}$, $c_2^{-1}p_{-1}(c_2)=b_{-1}$.

Thus one obtains an isomorphism $\bar{\phi}:\frac{\mathbb{O}_{\alpha}\cap\mathfrak{g}[u,u^{-1}]}{(\mathbb{O}_{\alpha}\cap\mathfrak{g}[u,u^{-1}])^{\perp}}\longrightarrow \mathfrak{g}\oplus\mathfrak{g}$. This implies that we have a 1-1 correspondence between Lagrangian subalgebras $W$ contained in $\mathbb{O}_{\alpha}\cap\mathfrak{g}[u,u^{-1}]$ and Lagrangian subalgebras $\bar{W}=\phi(W)$ in $\mathfrak{g}\oplus\mathfrak{g}$. Moreover $W$ is transversal to $\mathfrak{g}[[u]]$ if and only if $\bar{W}$ is transversal to $\phi(\mathbb{O}_{\alpha}\cap\mathfrak{g}[u])$. On the other hand, $\phi(\mathbb{O}_{\alpha}\cap\mathfrak{g}[u])$ consists of all pairs of the form $(a_0+c_1^{-1}b_0+b_1,a_0+c_2^{-1}b_0+b_1)$, where $a_0\in\mathfrak{g}_0$, $b_0,b_1\in\mathfrak{g}_{-1}$, which is precisely $\Delta_{\alpha}$. We note that $\mathfrak{g}_0$ is the reductive part of $P_{\alpha}^{-}=\mathfrak{g}_0+\mathfrak{g}_{-1}$. The proof is complete.

\end{proof}

\begin{rem}\label{rem1_A1}
The classification of all Lagrangian subalgebras $\bar{W}$ of $\mathfrak{g}\oplus\mathfrak{g}$ transversal to $\Delta_{\alpha}$ was accomplished in \cite{KPSST}, Theorem 11. This result states that, up to a conjugation which preserves 
$\Delta_{\alpha}$, a subalgebra $\bar{W}$ can be obtained from a pair formed by
a triple $(\Gamma_1,\Gamma_2,A)$ and a tensor $r\in\mathfrak{h}\otimes\mathfrak{h}$ such that $(\Gamma_1,\Gamma_2,A)$ is of type I or II and $r$ satisfies certain equations depending of the case. 
We recall that $\Gamma_1$, $\Gamma_2$ are subsets of $\Gamma$ and $A$ is an isometry between them. 
A triple is of type I if $\alpha\notin \Gamma_2$ and $(\Gamma_1,\theta(\Gamma_2),\theta(A))$ is admissible in the sense of Belavin-Drinfeld (see \cite{BD}). A triple is of type II if $\alpha\in \Gamma_2$,
$A(\beta)=\alpha$ and $(\Gamma_1\setminus\{\beta\},\theta(\Gamma_2\setminus\{\alpha\}),\theta(A))$ is admissible in the sense of Belavin-Drinfeld. Here $\theta$
denotes the Cartan involution of $\mathfrak{g}_0$. 

\end{rem}

\begin{rem}\label{rem2_A1}
The classification of Lagrangian subalgebras $\bar{W}$ can also be given using triples of the form 
$(\Gamma_1,\Gamma_2,A)$, where  
$\Gamma_{1}\subseteq\Gamma^{ext} \setminus \{\alpha\}$, 
$\Gamma_{2}\subseteq\Gamma$ and $A:\Gamma_1\longrightarrow \Gamma_2$ is an isometry. In this presentation, the case $k=1$ can be considered a particular case 
of the classification result which will be given in Remark \ref{generalized BD}.   
\end{rem}

Consider the Lie algebra 
$\mathfrak{g}[\varepsilon]$ with $\varepsilon^2=0$, endowed with the following invariant form: 
\[\bar{Q}_{\varepsilon}(x_1+\varepsilon x_2,y_1+\varepsilon y_2)=K(x_1,y_2)+K(x_2,y_1),\] 
for any elements $x_1$, $y_1$,  $x_2$, $y_2$ of $\mathfrak{g}$.
\begin{thm}\label{A2}
Let $\alpha$ be a simple root and $k$ its coefficient in the decomposition of 
$\alpha_{\rm{max}}$. Let $a(u)=\frac{1}{(1-u)^2}$.  

(i) If $k=1$, there exists a one-to-one correspondence between Lagrangian subalgebras $W$ of $\mathfrak{g}((u))$, with respect to $Q_{a(u)}$, which are transversal to $\mathfrak{g}[[u]]$ and satisfy 
$W \subseteq \mathbb{O}_{\alpha}\cap\mathfrak{g}[u,u^{-1}]$, and Lagrangian subalgebras in $\mathfrak{g}[\varepsilon]$, with respect to $\bar{Q}_{\varepsilon}$, transversal to $P_{\alpha}^{-}+\varepsilon (P_{\alpha}^{-})^{\perp}$. 

(ii) If $k>1$ then there are no Lagrangian subalgebras $W$ of $\mathfrak{g}((u))$, with respect to $Q_{a(u)}$, which are transversal to $\mathfrak{g}[[u]]$ and satisfy $W \subseteq \mathbb{O}_{\alpha}\cap\mathfrak{g}[u,u^{-1}]$. 

\end{thm}

\begin{proof}
We will only prove (i) (for (ii) see \cite{SZ}). Because $k=1$, one has the following: $\mathbb{O}_{\alpha}\cap\mathfrak{g}[u,u^{-1}]=u^{-1}\mathfrak{g}_1[u^{-1}]+\mathfrak{g}_0[u^{-1}]+
u\mathfrak{g}_{-1}[u^{-1}]$ and 
$(\mathbb{O}_{\alpha}\cap\mathfrak{g}[u,u^{-1}])^{\perp}=(u^{-1}-1)^2(u^{-1}\mathfrak{g}_1[u^{-1}]+\mathfrak{g}_0[u^{-1}]+
u\mathfrak{g}_{-1}[u^{-1}])$. 

Let us construct an isomorphism $\bar{\phi}$ between 
$\frac{\mathbb{O}_{\alpha}\cap\mathfrak{g}[u,u^{-1}]}{(\mathbb{O}_{\alpha}\cap\mathfrak{g}[u,u^{-1}])^{\perp}}$ and $\mathfrak{g}[\varepsilon]$.

We consider $\phi: \mathbb{O}_{\alpha}\cap\mathfrak{g}[u,u^{-1}]\longrightarrow \mathfrak{g}[\varepsilon]$ given by 
$\phi(u^{-1}p_1(u^{-1})+p_0(u^{-1})+up_{-1}(u^{-1}))=(1+\varepsilon)p_1(1+\varepsilon)+p_0(1+\varepsilon)+(1-\varepsilon)p_{-1}(1+\varepsilon)$, 
where $p_1\in \mathfrak{g}_1[u^{-1}]$, $p_0\in\mathfrak{g}_0[u^{-1}]$ and 
$p_{-1}\in\mathfrak{g}_{-1}[u^{-1}]$. One can check that the kernel of this map 
is the ideal generated by $(u^{-1}-1)^2$. The map $\phi$ is also surjective. 
Indeed, let us write $a+\varepsilon b=(a_1+a_0+a_{-1})+\varepsilon (b_1+b_0+b_{-1})$. 
One can uniquely find first degree polynomials $p_1$, $p_0$ and $p_{-1}$ such that $\phi(u^{-1}p_1(u^{-1})+p_0(u^{-1})+up_{-1}(u^{-1}))=a+\varepsilon b$. Straightforward computations give: 
$p_1(u^{-1})=(2a_1-b_1)+u^{-1}(b_1-a_1)$, $p_0(u^{-1})=a_0-b_0+u^{-1}b_0$, $p_{-1}(U^{-1})=-b_{-1}+u^{-1}(a_{-1}+b_{-1})$. 

By means of the above isomorphism, we have a 1-1 correspondence between Lagrangian subalgebras $W$ of $\mathfrak{g}((u))$, with respect to $Q_{a(u)}$, which are transversal to $\mathfrak{g}[[u]]$ and satisfy 
$W \subseteq \mathbb{O}_{\alpha}\cap\mathfrak{g}[u,u^{-1}]$, and Lagrangian subalgebras in $\mathfrak{g}[\varepsilon]$, with respect to $\bar{Q}_{\varepsilon}$, transversal to $\phi(\mathbb{O}_{\alpha}\cap\mathfrak{g}[u])$. 

On the other hand, $\phi(\mathbb{O}_{\alpha}\cap\mathfrak{g}[u])=\phi(
\mathfrak{g}_0+u(\mathfrak{g}_{-1}+u^{-1}\mathfrak{g}_{-1})$ which consists of elements of the form $a_0+(1-\varepsilon)b_{-1}+c_{-1}$, for all $a_0\in \mathfrak{g}_0$ and $b_{-1},c_{-1}\in \mathfrak{g}_{-1}$. Since, $P_{\alpha}^{-}=\mathfrak{g}_0+\mathfrak{g}_{-1}$, it follows that $\phi(\mathbb{O}_{\alpha}\cap\mathfrak{g}[u])$ is precisely 
$P_{\alpha}^{-}+\varepsilon (P_{\alpha}^{-})^{\perp}$. This ends the proof.

\end{proof}

\begin{rem}\label{rem_A2}
In \cite{S} it was shown that Lagrangian subalgebras of $\mathfrak{g}[\varepsilon]$, with respect to $\bar{Q}_{\varepsilon}$, and transversal to $P_{\alpha}^{-}+\varepsilon (P_{\alpha}^{-})^{\perp}$ are in a 1-1 correspondence with pairs $(L,B)$, where $L$ is a subalgebra of $\mathfrak{g}$ satisfying $L+P_{\alpha}^{-}=\mathfrak{g}$ and $B$ is a 2-cocycle on $L$ nondegenerate on $L\cap P_{\alpha}^{-}$. 

\end{rem}

Let us assume again that $\alpha$ is a simple root with coefficient $k$ in
$\alpha_{\rm{max}}$. Denote by $L_{\alpha}$ the Lie subalgebra of $\mathfrak{g}$ whose Dynkin diagram is obtained from the extended Dynkin diagram of $\mathfrak{g}$ by erasing $\alpha$. Then the restriction of the bilinear form $\bar{Q}$ on $\mathfrak{g}\oplus\mathfrak{g}$ is nondegenerate on $L_{\alpha}\times \mathfrak{g}$. 

Consider the standard parabolic subalgebra of $L_{\alpha}$ corresponding to $-\alpha_{\rm{max}}$, $P_{\alpha_{\rm{max}}}^{+}$. We note that $P_{\alpha_{\rm{max}}}^{+}$ and $P_{\alpha}^{-}$ have equal reductive components.
Let $\Delta_{\alpha,\alpha_{\rm{max}}}$ be the set of pairs $(x,y)\in P_{\alpha_{\rm{max}}}^{+}\times P_{\alpha}^{-}$ with equal reductive parts.  

\begin{thm}\label{A3}
Let $\alpha$ be a simple root and $k$ its coefficient in the decomposition of 
$\alpha_{\rm{max}}$. Let $a(u)=\frac{1}{1-u}$. 

There exists a one-to-one correspondence between Lagrangian subalgebras $W$ of $\mathfrak{g}((u))$, with respect to $Q_{a(u)}$, which are transversal to $\mathfrak{g}[[u]]$ and satisfy 
$W \subseteq \mathbb{O}_{\alpha}\cap\mathfrak{g}[u,u^{-1}]$, and Lagrangian subalgebras in $L_{\alpha}\oplus\mathfrak{g}$, with respect to $\bar{Q}$, transversal to $\Delta_{\alpha,\alpha_{\rm{max}}}$. 
\end{thm}

\begin{proof}
For each $r$, $-k\leq r\leq k$, let $R_{r}$ denote the set of all roots which contain $\alpha$ with coefficient $r$. Let $\mathfrak{g}_{0}=\mathfrak{h}\oplus\sum_{\beta\in R_{0}}\mathfrak{g}_\beta$ and $\mathfrak{g}_r=\sum_{\beta\in R_r}\mathfrak{g}_\beta$. Then
\[
\mathbb{O}_\alpha=\sum_{r=1}^k u^{-1}\mathfrak{g}_r[[u^{-1}]]+
\sum_{r=1-k}^0 \mathfrak{g}_r[[u^{-1}]]+u\mathfrak{g}_{-k}[[u^{-1}]],\]

\[\mathbb{O}_{\alpha}\cap\mathfrak{g}[u,u^{-1}]=\sum_{r=1}^k u^{-1}\mathfrak{g}_r[u^{-1}]+
\sum_{r=1-k}^0 \mathfrak{g}_r[u^{-1}]+u\mathfrak{g}_{-k}[u^{-1}],\] 
\[(\mathbb{O}_{\alpha}\cap\mathfrak{g}[u,u^{-1}])^{\perp}=(1-u)(u^{-3}\mathfrak{g}_k
[u^{-1}]+\sum_{r=0}^{k-1}u^{-2}\mathfrak{g}_r[u^{-1}]+\sum_{r=-k}^{-1}u^{-1} \mathfrak{g}_r[u^{-1}]).\]

Let $\phi: \mathbb{O}_{\alpha}\cap\mathfrak{g}[u,u^{-1}]\longrightarrow (\mathfrak{g}_k+\mathfrak{g}_0+\mathfrak{g}_{-k})\oplus \mathfrak{g}$ be defined by  
\[ \phi(\sum_{r=1}^k u^{-1}p_r(u^{-1})+\sum_{r=1-k}^0p_r(u^{-1})+up_{-k}(u^{-1}))=\]
\[(p_k(0)+p_0(0)+p_{-k}(0),\sum_{r=1}^kp_r(1)+\sum_{r=1-k}^0p_r(1)+p_{-k}(1)).\]
One can check that $\phi$ is an epimorphism whose kernel is $\mathbb{O}_{\alpha}\cap\mathfrak{g}[u,u^{-1}])^{\perp}$. By means of this morphism, we have a 1-1 correspondence between Lagrangian subalgebras $W$ of $\mathfrak{g}((u))$, contained in $\mathbb{O}_{\alpha}\cap\mathfrak{g}[u,u^{-1}]$, and Lagrangian subalgebras $\bar{W}$ of $(\mathfrak{g}_k+\mathfrak{g}_0+\mathfrak{g}_{-k})\oplus \mathfrak{g}$. We observe that $\mathfrak{g}_k+\mathfrak{g}_0+\mathfrak{g}_{-k}=L_{\alpha}$.
On the other hand, $\mathbb{O}_{\alpha}\cap\mathfrak{g}[u]=\sum_{r=1-k}^0\mathfrak{g}_r+u(\mathfrak{g}_{-k}+u^{-1}\mathfrak{g}_{-k})$. Its image in 
$L_{\alpha}\oplus\mathfrak{g}$ via $\phi$ consists of elements of the form
$(a_0+a_{-k},a_0+ \sum_{r=-k}^{-1}a_r+b_{-k})$, where $a_r\in\mathfrak{g}_r$, $a_{-k},b_{-k}\in\mathfrak{g}_{-k}$. Since $P_{\alpha_{\rm{max}}}^{+}=\mathfrak{g}_0+\mathfrak{g}_{-k}$ and $P_{\alpha}^{-}=\mathfrak{g}_0+\mathfrak{g}_1+...+\mathfrak{g}_{-k}$, this set coincides with $\Delta_{\alpha,\alpha_{\rm{max}}}$. The proof is now complete.
\end{proof}

\begin{rem}\label{generalized BD}
The classification of all Lagrangian subalgebras $\bar{W}$ of $\mathfrak{g}\oplus\mathfrak{g}$ transversal to $\Delta_{\alpha,\alpha_{\rm{max}}}$ was accomplished in \cite{PS}, Theorem 2.13. This result states that, up to a conjugation which preserves 
$\Delta_{\alpha,\alpha_{\rm{max}}}$, a subalgebra $\bar{W}$ can be obtained from a pair formed by
a triple $(\Gamma_1,\Gamma_2,A)$ of type I or II and a subspace
$\mathfrak{i}_{\mathfrak{a}}$ satisfying a certain condition. Let us recall how types I and II are defined. 
Let $i$ denote the embedding of $\mathfrak{g}_0$ into $L_{\alpha}$.
Here one considers triples of the form $(\Gamma_{1},\Gamma_{2},A)$, where $\Gamma_{1}\subseteq\Gamma^{ext} \setminus \{\alpha\}$, 
$\Gamma_{2}\subseteq\Gamma$ and $A$ is an isometry between $\Gamma_{1}$ and $\Gamma_{2}$. 
A triple $(\Gamma_{1},\Gamma_{2},A)$
is called of \textit{type I} if $\alpha\notin\Gamma_{2}$ and $(\Gamma_{1},i(\Gamma_{2}),iA)$ is an admissible triple in the sense of Belavin--Drinfeld.
The triple $(\Gamma_{1},\Gamma_{2},A)$ is called of \textit{type II} if $\alpha\in\Gamma_{2}$ and 
$A(\beta)=\alpha$, for some
$\beta\in\Gamma_{1}$ and $(\Gamma_{1}\setminus\{\beta\},
i(\Gamma'_{2}\setminus\{\alpha\}),iA')$ is an
admissible triple in the sense of Belavin--Drinfeld.
The space $\mathfrak{i}_{\mathfrak{a}}$ is a Lagrangian subspace of
$\mathfrak{a}:=\{(h_{1},h_{2})\in\mathfrak{h}\times\mathfrak{h}:\beta(h_{1})=0,
\gamma(h_{2})=0, \forall\beta\in\Gamma_{1}, \forall \gamma\in\Gamma_{2}\}$.

\end{rem}

Finally, the remaining case $a(u)=1$ follows by reformulating Prop. 3.2.1. from \cite{S}. The result is the following:

Let $\alpha$ be a simple root with coefficient $k$. There exists a 1-1 correspondence between  Lagrangian subalgebras $W$ of $\mathfrak{g}((u))$, with respect to $Q_{a(u)}$, which are transversal to $\mathfrak{g}[[u]]$ and satisfy 
$W \subseteq \mathbb{O}_{\alpha}\cap\mathfrak{g}[u,u^{-1}]$, and Lagrangian subalgebras $\bar{W}$ in $(L_{\alpha}+\varepsilon^{k} L_{\alpha})\oplus (\oplus_r \varepsilon^{r}V_{\alpha,r})$ transversal to $(P_{\alpha}^{-}+\varepsilon^{k}(P_{\alpha}^{-})^{\perp})\oplus(\oplus_r\varepsilon^rP_{\alpha,r}^{-})$ (see notation in \cite{S}, p. 537).

This concludes the analysis of the Lie bialgebra structures in case I. 

\section{Lie bialgebra structures on  $\mathfrak{g}[u]$ in case II}

 Consider $\mathfrak{g}((u))\oplus\mathfrak{g}$ endowed with
\[ Q_{a(u)}(f_1(u)+x_1,f_2(u)+x_2)=\mathrm{Res}_{u=0}(u^{-1}a(u)K(f_1(u),f_2(u)))\]\[-K(x_1,x_2),\]
for all $f_1(u), f_2(u)\in\mathfrak{g}((u))$ and $x_1,x_2\in\mathfrak{g}$. According to \cite{SZ}, the following statement holds: 
\begin{prop}
There exists a one-to-one correspondence between Lie bialgebra structures 
$\delta$ on $\mathfrak{g}[u]$ satisfying  
$D_{\bar{\delta}}(\mathfrak{g}[[u]])=\mathfrak{g}((u))\oplus\mathfrak{g}$ and 
bounded Lagrangian subalgebras $W$ of $\mathfrak{g}((u))\oplus\mathfrak{g}$, 
with respect to 
the nondegenerate bilinear form $Q_{a(u)}$, and transversal to $\mathfrak{g}[[u]]$.

\end{prop}

For any $\sigma\in\mathrm{Aut}_{\mathbb{C}[u]}(\mathfrak{g}[u])$, denote by  $\tilde{\sigma}(u)=\sigma(u)\oplus\sigma(0)$, regarded as an automorphism of $\mathfrak{g}((u))\oplus\mathfrak{g}$. According to \cite{PM}, the following result holds:

\begin{prop}\label{max_ord2}
Suppose that $W$ is a bounded Lagrangian subalgebra of $\mathfrak{g}((u))\oplus\mathfrak{g}$, with respect to $Q_{a(u)}$ and transversal to $\mathfrak{g}[[u]]$. 
Then there exists $\sigma\in\mathrm{Aut}_{\mathbb{C}[u]}(\mathfrak{g}[u])$ such that $\tilde{\sigma}(u)(W)\subseteq (\mathbb{O}_{\alpha}\cap \mathfrak{g}[u,u^{-1}])\oplus\mathfrak{g}$, where $\alpha$ is either a simple root or $-\alpha_{\max}$. 
\end{prop}

\begin{thm}\label{B1}
Let $\alpha$ be a simple root and $k$ its coefficient in the decomposition of 
$\alpha_{\rm{max}}$. Let $a(u)=\frac{1}{1-u}$. 

(i) If $k=1$, there exists a one-to-one correspondence between Lagrangian subalgebras $W$ of $\mathfrak{g}((u))\oplus\mathfrak{g}$, with respect to $Q_{a(u)}$, which are transversal to $\mathfrak{g}[[u]]$ and satisfy 
$W \subseteq (\mathbb{O}_{\alpha}\cap\mathfrak{g}[u,u^{-1}])\oplus\mathfrak{g}$, and Lagrangian subalgebras in $\mathfrak{g}\oplus\mathfrak{g}$, with respect to $\bar{Q}$, transversal to $\Delta_{\alpha}$. 

(ii) If $k>1$, there are no Lagrangian subalgebras $W$ of $\mathfrak{g}((u))\oplus\mathfrak{g}$, with respect to $Q_{a(u)}$, which are transversal to $\mathfrak{g}[[u]]$ and satisfy 
$W \subseteq (\mathbb{O}_{\alpha}\cap\mathfrak{g}[u,u^{-1}])\oplus\mathfrak{g}$.\end{thm}

\begin{proof}
If $k=1$, then $\mathbb{O}_{\alpha}\cap\mathfrak{g}[u,u^{-1}]=u^{-1}\mathfrak{g}_1[u^{-1}]+\mathfrak{g}_0[u^{-1}]+
u\mathfrak{g}_{-1}[u^{-1}]$ and 
$(\mathbb{O}_{\alpha}\cap\mathfrak{g}[u,u^{-1}]\oplus\mathfrak{g})^{\perp}=(u^{-1}-1)(u^{-1}\mathfrak{g}_1[u^{-1}]+\mathfrak{g}_0[u^{-1}]+
u\mathfrak{g}_{-1}[u^{-1}])$. Then there exists an epimorphism 
\[\phi: \mathbb{O}_{\alpha}\cap\mathfrak{g}[u,u^{-1}]\oplus\mathfrak{g}\longrightarrow \mathfrak{g}\oplus\mathfrak{g}\] given by 
\[\phi((u^{-1}p_1(u^{-1})+p_0(u^{-1})+up_{-1}(u^{-1}),a)=p_1(1)+p_0(1)+p_{-1}(1),a),\]for all $p_1(u^{-1})\in\mathfrak{g}_1[u^{-1}]$, $p_0(u^{-1})\in\mathfrak{g}_0[u^{-1}]$, $p_{-1}(u^{-1})\in\mathfrak{g}_{-1}[u^{-1}]$ and $a\in\mathfrak{g}$. Obviously the kernel of $\phi$ is $(\mathbb{O}_{\alpha}\cap\mathfrak{g}[u,u^{-1}]\oplus\mathfrak{g})^{\perp}$. Let us denote by $\bar{\phi}$ the isomorphism induced by $\phi$ between the quotient and $\mathfrak{g}\oplus\mathfrak{g}$. One can easily check that $\bar{\phi}((\mathbb{O}_{\alpha}\oplus\mathfrak{g})\cap\mathfrak{g}[u])=\Delta_{\alpha}$
and thus we have a correspondence between $W$ and Lagrangian subalgebras $\bar{W}$ of $\mathfrak{g}\oplus\mathfrak{g}$ transversal to  $\Delta_{\alpha}$. 

Statement (ii) was proved in \cite{SZ}. This ends the proof.
\end{proof}

\begin{rem}
We see that there is an analogy between the above result and Theorem \ref{A1}. 
Therefore Lagrangian subalgebras $W$ in $\mathfrak{g}((u))\oplus\mathfrak{g}$ with the required properties can be expressed using the data given in Remark \ref{rem1_A1}
and \ref{rem2_A1}.
\end{rem}
The remaining case to be discussed is $a(u)=1$. 
\begin{thm}\label{B2}
Suppose $a(u)=1$ and let $\alpha$ be a simple root and $k$ its coefficient in the decomposition of $\alpha_{\rm{max}}$. Then there exists a one-to-one correspondence between Lagrangian subalgebras $W$ of $\mathfrak{g}((u))\oplus\mathfrak{g}$, with respect to $Q_{a(u)}$, which are transversal to $\mathfrak{g}[[u]]$ and satisfy 
$W \subseteq (\mathbb{O}_{\alpha}\cap\mathfrak{g}[u,u^{-1}])\oplus\mathfrak{g}$
and Lagrangian subalgebras in $L_{\alpha}\oplus\mathfrak{g}$, with respect to $\bar{Q}$, transversal to $\Delta_{\alpha,\alpha_{\rm{max}}}$. 

\end{thm}

\begin{proof}
 We will use the same notation as in the proof of Theorem \ref{A3}. We recall the following: 

\[\mathbb{O}_{\alpha}\cap\mathfrak{g}[u,u^{-1}]=\sum_{r=1}^k u^{-1}\mathfrak{g}_r[u^{-1}]+
\sum_{r=1-k}^0 \mathfrak{g}_r[u^{-1}]+u\mathfrak{g}_{-k}[u^{-1}],\] 
\[(\mathbb{O}_{\alpha}\cap\mathfrak{g}[u,u^{-1}]\oplus\mathfrak{g})^{\perp}=
\sum_{-k}^{-1}\mathfrak{g}_r[[u^{-1}]]+\sum_{r=0}^{k-1}u^{-1}\mathfrak{g}_r[[u^{-1}]]+u^{-2}\mathfrak{g}_k.\]
Then we have an isomorphism 
\[\bar{\phi}: \frac{\mathbb{O}_{\alpha}\cap\mathfrak{g}[u,u^{-1}]\oplus\mathfrak{g}}{(\mathbb{O}_{\alpha}\cap\mathfrak{g}[u,u^{-1}]\oplus\mathfrak{g})^{\perp}}
\longrightarrow (\mathfrak{g}_k+\mathfrak{g}_0+\mathfrak{g}_{-k})\oplus\mathfrak{g}\] 
such that for the equivalence class of a pair $(f,x)\in\mathbb{O}_{\alpha}\cap\mathfrak{g}[u,u^{-1}]\oplus\mathfrak{g}$, $\bar{\phi}(f,x)=(a_0+b_0+c_0,x)$, 
if $f=u^{-1}(a_0+a_1u^{-1}+...)+(b_0+b_1u^{-1}+...)+u(c_0+c_1u^{-1}+...)+...$,
$a_i\in\mathfrak{g}_k$, $b_i\in\mathfrak{g}_0$, $c_i\in\mathfrak{g}_{-k}$, $x\in\mathfrak{g}$. 
The image of $(\mathbb{O}_{\alpha}\oplus\mathfrak{g})\cap\mathfrak{g}[u]$ is eaxctly $\Delta_{\alpha,\alpha_{\rm{max}}}$. 

\end{proof}

\begin{rem}
With a different formulation, this result also appeared in \cite{PS}, where the so-called quasi-trigonometric solutions of the CYBE were classified.
\end{rem}
\begin{rem}
We also note the analogy between this theorem and Theorem \ref{A3}. The corresponding Lagrangian subalgebras can be described using the data given in Remark \ref{generalized BD}.
\end{rem}

\section{Lie bialgebra structures on  $\mathfrak{g}[u]$ in case III}
Finally, let us treat case III, where the associated double is $\mathfrak{g}((u))\oplus \mathfrak{g}[\varepsilon]$, 
with the nondegenerate bilinear form given by the formula

\[Q_{a(u)}(f_1(u)+x_2+\varepsilon x_3,f_2(u)+y_2+\varepsilon y_3)=\mathrm{Res}_{u=0}(u^{-2}a(u)K(f_1(u),f_2(u))-\]
\[-K(x_3,y_2)-K(x_2,y_3),\]
for any $f_1(u),f_2(u)\in\mathfrak{g}((u))$ and $x_2,x_3,y_2,y_3\in\mathfrak{g}$. According to \cite{SZ}, the following result holds:
\begin{prop}
There exists a one-to-one correspondence between Lie bialgebra structures 
$\delta$ on $\mathfrak{g}[u]$ satisfying  
$D_{\bar{\delta}}(\mathfrak{g}[[u]])=\mathfrak{g}((u))\oplus \mathfrak{g}[\varepsilon]$ and 
bounded Lagrangian subalgebras $W$ of $\mathfrak{g}((u))\oplus \mathfrak{g}[\varepsilon]$, with respect to 
the nondegenerate bilinear form $Q_{a(u)}$, and transversal to $\mathfrak{g}[[u]]$.
\end{prop}

Recall that any $\sigma(u)\in \mathrm{Ad}(\mathfrak{g}[u])$  induces an automorphism $\sigma(0)\in\mathrm{Ad}(\mathfrak{g})$, which in turn gives an well-defined automorphism $\bar{\sigma}(0)$ of $\mathfrak{g}[\varepsilon]$ via $\bar{\sigma}(0)(x+\varepsilon y)=\sigma(0)(x)+\varepsilon\sigma(0)(y)$. Then $\tilde{\sigma}(u)=\sigma(u)\oplus \bar{\sigma}(0)$ is an automorphism of $\mathfrak{g}((u))\oplus\mathfrak{g}[\varepsilon]$. 

\begin{prop}\label{max_ord3}\cite{PM}
Suppose that $W$ is a bounded Lagrangian subalgebra of $\mathfrak{g}((u))\oplus\mathfrak{g}[\varepsilon]$, with respect to $Q_{a(u)}$ and transversal to $\mathfrak{g}[[u]]$. 
Then there exists $\sigma\in\mathrm{Ad}_{\mathbb{C}[u]}(\mathfrak{g}[u])$ such that $\tilde{\sigma}(u)(W)\subseteq (\mathbb{O}_{\alpha}\cap \mathfrak{g}[u,u^{-1}])\oplus\mathfrak{g}[\varepsilon]$, where $\alpha$ is either a simple root or $-\alpha_{\max}$. 

\end{prop}

The case $-\alpha_{\max}$ has already been analysed in \cite{PM}. For an arbitrary simple root we have: 
\begin{thm}\label{C1}
Suppose $a(u)=1$ and let $\alpha$ be a simple root and $k$ its coefficient in the decomposition of $\alpha_{\rm{max}}$. 

(i) If $k=1$, there exists a one-to-one correspondence between Lagrangian subalgebras $W$ of $\mathfrak{g}((u))\oplus\mathfrak{g}[\varepsilon]$, with respect to $Q_{a(u)}$, which are transversal to $\mathfrak{g}[[u]]$ and satisfy 
$W \subseteq (\mathbb{O}_{\alpha}\cap\mathfrak{g}[u,u^{-1}])\oplus\mathfrak{g}[\varepsilon]$
and Lagrangian subalgebras in $\mathfrak{g}[\varepsilon]$, with respect to $\bar{Q}_{\varepsilon}$, transversal to $P_{\alpha}^{-}+\varepsilon (P_{\alpha}^{-})^{\perp}$. 

(ii)  If $k>1$ there are no Lagrangian subalgebras $W$ of $\mathfrak{g}((u))\oplus\mathfrak{g}[\varepsilon]$, with respect to $Q_{a(u)}$, which are transversal to $\mathfrak{g}[[u]]$ and satisfy 
$W \subseteq (\mathbb{O}_{\alpha}\cap\mathfrak{g}[u,u^{-1}])\oplus\mathfrak{g}[\varepsilon]$. 
\end{thm}

\begin{proof} 
(ii) was proved in \cite{SZ}. (i) If $k=1$,  
 $\mathbb{O}_{\alpha}\cap\mathfrak{g}[u,u^{-1}]=u^{-1}\mathfrak{g}_1[u^{-1}]+\mathfrak{g}_0[u^{-1}]+
u\mathfrak{g}_{-1}[u^{-1}]$ and 
$(\mathbb{O}_{\alpha}\cap\mathfrak{g}[u,u^{-1}]\oplus\mathfrak{g}[\varepsilon])^{\perp}=\mathbb{O}_{\alpha}\cap\mathfrak{g}[u,u^{-1}]$. Then $\frac{\mathbb{O}_{\alpha}\cap\mathfrak{g}[u,u^{-1}]\oplus\mathfrak{g}[\varepsilon]}{(\mathbb{O}_{\alpha}\cap\mathfrak{g}[u,u^{-1}]\oplus\mathfrak{g}[\varepsilon])^{\perp}}$ is obviously isomorphic to $\mathfrak{g}[\varepsilon]$. The image of $(\mathbb{O}_{\alpha}\oplus\mathfrak{g}[\varepsilon])\cap \mathfrak{g}[u]$
via this isomorphism is $P_{\alpha}^{-}+\varepsilon (P_{\alpha}^{-})^{\perp}$. The conclusion follows immediately.
\end{proof}

\begin{rem}
This result was also obtained in \cite{SY}, where quasi-rational $r$-matrices were studied. 
\end{rem}
\begin{rem}
We observe that the above theorem is analogous to Theorem \ref{A2} and Remark \ref{rem_A2} also holds. 
\end{rem}
\bibliographystyle{amsalpha}

\end{document}